\documentclass[12pt, a4paper, twosided, reqno]{amsart}
\usepackage{enumerate, cite}
\usepackage{hyperref}

\newtheorem{theorem}{Theorem}
\newtheorem{lemma}{Lemma}

\newtheorem{example}{Example}

\allowdisplaybreaks
\author[N. Gahlian ]{Nidhi Gahlian }

\address{nidhi gahlian; department of mathematics, university of delhi, delhi-110007, india.}
\email{nidhigahlyan81@gmail.com}

\thanks {Research work of the  author is supported by research fellowship from Department of Science and Technology(INSPIRE), New Delhi, India, IF-190674.}
\title[Meromorphic Solutions]{Existence and Nonexistence of  Solutions of Certain Type of Nonlinear Differential and  Differential-Difference Equations}
\begin{document}
	\maketitle
	
	\begin{abstract}
		In this paper, we study the existence, nonexistence, and growth behaviour of solutions of certain nonlinear differential and differential–difference equations involving the term $f^{n}f'$. Using tools from Nevanlinna theory, we establish several nonexistence results for finite-order entire solutions of equations involving $ff'$. In particular, we describe all the possible entire solutions  of the type $f^{2}f'$ in which right hand side involves exponential polynomials. The results obtained by us  improve several known results in the literature and we have  them demonstrated through illustrative examples.
		
	\end{abstract}
	\section{\textbf{Introduction}}
The study of entire and meromorphic solutions of nonlinear differential equations has become an active and significant area of research in complex analysis due to its rich theoretical framework and wide range of applications. In $1964$, Hayman's extension \cite{hay} of the Tumura–Clunie theorem initiated extensive research on differential equations of Tumura–Clunie type, especially those represented by $f^n(z) + Q_d(z,f) = g(z) $, where  $Q_d(z,f)$ denotes a differential polynomial in $f$. Since then, numerous authors have investigated these equations under a variety of perturbations and structural assumptions, leading to a rich collection of results concerning the growth and value distribution of their solutions. In $2004$, Yang and Li \cite{yangcc} investigated the following nonlinear differential equation 
	 	\begin{equation*}
	 	4f^3+3f''=-\sin 3z,
	 \end{equation*} and proved that it admits exactly three non-constant entire solutions.

Since the term $-\sin 3z $ can be represented as a linear combination of exponential functions, it naturally motivates the investigation of differential equations whose forcing terms are composed of exponential functions. Inspired by this observation, considerable attention has been devoted to the study of nonlinear differential equations of the following form:

	\begin{equation}\label{..}
		f^n(z)+Q_d(z,f)=p_1(z)e^{\alpha_1 (z)}+p_2(z)e^{\alpha_2(z)},
	\end{equation}
	where $p_1(z), p_2(z)$ are rational functions, $\alpha_1(z)$, $\alpha_2(z)$ are polynomials, and $ Q_d(z,f)$ is a differential polynomial of degree $d$. 
	
	In $2006$, Li and Yang\cite{lip} established results concerning the existence of entire solutions of equation \eqref{..}. In $2011$,  Li \cite{li2} investigated the corresponding meromorphic solutions and identified their forms for specific forms of  $\alpha_1 (z)$ and  $ \alpha_2(z) $.  Building on these findings, in $2013$, Liao et al. \cite{liao} established exact meromorphic solutions of equation \eqref{..}.
	Furthermore, this line of research has been extended by replacing $f^n$ with $f^nf'$ in L.H.S. and R.H.S., by a single term $u(z)e^{v(z)}$ thereby broadening the scope and depth of investigation in this area.  In $2014$, Liao and Ye\cite{jl} gave the following theorem.
	
	\begin{theorem}\cite{jl}
		Let $Q_d(z,f)$ be a differential polynomial in $f(z)$ of degree $d$ with rational function
		coefficients. Suppose that $u(z)$ is a nonzero rational function and $v(z)$ is a nonconstant polynomial. If $d\leq n-1$ and the differential equation
		\begin{equation}
				f^nf' + Q_d(z,f )=u(z)e^{v(z)}
		\end{equation}
	admits a meromorphic solution $f$ with finitely many poles, then $f$ has the following form:
	$f(z)=s(z)e^{\frac{v(z)}{n+1}}$ and $Q_d(z,f)\equiv 0$,
		where $s(z)$ is a rational function with $s^n((n+1)s' + v's)=(n +1)u$.
		\end{theorem}
	In $2017$, Zhang et al. \cite{zh}	investigated the following differential equation
	\begin{equation}\label{;}
			f^nf ' + Q_d(z,f )=p_1(z)e^{\alpha_1z}+p_2(z)e^{\alpha_2z},
	\end{equation}
	where  $Q_d(z,f)$ is a differential polynomial in $f(z)$ of degree $ d\leq  n-2$ with rational functions as its coefficients, $p_1(z), p_2(z)$ are non-vanishing rational functions and
$\alpha_1, \alpha_2$ are nonzero constants.
They  obtained the given result.
	\begin{theorem}\cite {yl}
		Let $n \geq 3$  be an integer and suppose that equation \eqref{;} admits a transcendental meromorphic solution $f(z)$ with finitely many poles
		then $\frac{\alpha_1}{\alpha_2}$ is a rational number and $f(z)$ must be of the form
		$f(z)=q(z)e^{P(z)}$, where $q(z)$ is a rational function and P$(z)$ is a nonconstant polynomial. Furthermore, only one of the following two cases holds:\\
		$(i)$ $\frac{\alpha_1}{\alpha_2}=1$, $Q_d(z,f) \equiv 0$ and $(n+1)P'(z)=\alpha_1= \alpha_2$;\\
		$(ii)$  $\frac{\alpha_1}{\alpha_2}=\frac{n+1}{k}$, $k$ is an integer, $1 \leq k \leq d$, $Q_d(z,f) \equiv p_2(z)e^{\alpha_2z}$ and
		$(n+1)P'(z)=\alpha_1$ or $\frac{\alpha_1}{\alpha_2}=\frac{k}{n+1}$, $k$ is an integer, $1 \leq k \leq d$, $Q_d(z,f) \equiv p_1(z)e^{\alpha_1z}$
	 and $(n +1)P'(z)=\alpha_2$.
		\end{theorem}
	The above theorem naturally leads to the investigation of several perturbative variants. One can see\cite{bo, yl} and references therein. \\
     It is natural to ask what can be said for the case when $n<3$ ? The following results provide a partial answer to this question.
	
		\begin{theorem}\label{main2}
		Suppose that $p_{1}, p_{2}, \gamma_{1},  \gamma_{2}$ are non-zero constants such that $\gamma_{1}\neq\gamma_{2}$, $U^*(z,f)$ is a linear difference polynomial in $f$ with small function of $f$ as its coefficients. If $f$ is a transcendental entire solution, with $\rho_2(f)<1$, of the difference equation
		\begin{equation}\label{mde}
			f^2(z)f'(z)+U^*(z,f)=p_1e^{\gamma_{1}z}+p_2e^{\gamma_{2}z},
		\end{equation}
		then $\rho(f)=1$ and any one of the following holds:
		\begin{enumerate}[(i)]
			\item $2T(r,f)\leq 3N\left(r,\frac{1}{f}\right)+S(r,f)$;
			\item $f(z)=c_1e^{\frac{\gamma_1 z}{3}}$, where $c_1=\left(\frac{3p_1}{\gamma_1}\right)^\frac{1}{3}$, $p_2=\sum_{j=0}^{m}d_jc_1e^{\tilde{\gamma}}$ and $\frac{\gamma_1}{\gamma_2}=3$;
			\item $f=c_2e^{\frac{\gamma_2z}{3}}$, where $c_2=\left(\frac{3p_2}{\gamma_2}\right)^\frac{1}{3}$, $p_1=\sum_{j=0}^{m}d_jc_2e^{\tilde{\gamma}}$ and $\frac{\gamma_2}{\gamma_1}=3$;
		\end{enumerate}
		where $c_1, c_2, \tilde{\gamma} \; \text{and} \; d_j $ $(j=0,1,2,...,m)$  are constants.
	\end{theorem}

	The  below mentioned  examples support our theorem.

	\begin{example}
		The function $f(z)=2e^z$ is the solution of the following difference equation
		\begin{equation}
			f^2f'+\frac{1}{4}f(z+\log 2)=8e^{3z}+e^z,
		\end{equation}
		where $\frac{\gamma_1}{3}=1=\gamma_2$ satisfies conclusion $(ii)$ of the theorem.
	\end{example}
	\begin{example}
		The function $f(z)=2e^{\frac{3z}{2}}$ is the solution of the following difference equation
		\begin{equation}
			f^2f'+\frac{1}{8}f(z+\log 4)=2e^{\frac{3z}{2}}+12e^{\frac{9z}{2}},
		\end{equation}
		where $\frac{\gamma_2}{3}=\frac{3}{2}=\gamma_1$, $c_2=2$ and
		\begin{equation*}
			p_1=d_1c_2e^{\tilde{\gamma}}=\frac{1}{8}2e^{\frac{3}{2}\log4}=2
		\end{equation*} satisfies conclusion $(iii)$ of the theorem.
	\end{example}
	
	\begin{example}
		The function $f(z)=e^z+1$ is the solution of the following difference equation
		\begin{equation}
			f^2f'+f(z)-f(z+\log 2)=e^{3z}+2e^{2z},
		\end{equation}
		where $U^*(z,f)=f(z)-f(z+\log2)$, satisfies conclusion $(i)$ of the theroem.
	\end{example}
	


		A linear differential polynomial $Q$  is an expression formed by finite sum of monomials  involving $f$ and its derivatives. It has the general form
		\begin{equation*}
			Q(z,f)=\sum_{ j=0}^{k}a_jf^{(j)},
		\end{equation*}
		where $a_j\; (j=0,1,...,k)$  are constants.
		However, if a linear differential polynomial is considered instead of a difference polynomial in Theroem \ref{main2}, then condition  $\rho_2(f)<1$  is no longer needed. Consequently, we have the following theorem for $n=2$.
		\begin{theorem}\label{main6}
			Suppose that $p_{1}, p_{2}, \gamma_{1},  \gamma_{2}$ are non-zero constants such that $\gamma_{1}\neq\gamma_{2}$, $Q(z,f)$ is a linear differential polynomial in $f$ with small function of $f$ as its coefficients. If $f$ is a transcendental entire solution of the differential  equation
			\begin{equation}\label{mde}
				f^2(z)f'(z)+Q(z,f)=p_1e^{\gamma_{1}z}+p_2e^{\gamma_{2}z},
			\end{equation}
			then $\rho(f)=1$ and any one of the following holds:
			\begin{enumerate}[(i)]
				\item $2T(r,f)\leq 3N\left(r,\frac{1}{f}\right)+S(r,f)$;
				\item $f(z)=c_1e^{\frac{\gamma_1 z}{3}}$, where $c_1=\left(\frac{3p_1}{\gamma_1}\right)^\frac{1}{3}$, $p_2=c_1\sum_{j=0}^{k}a_j\left(\frac{\gamma_1}{3}\right)^j$ and $\frac{\gamma_1}{\gamma_2}=3$;
				\item $f=c_2e^{\frac{\gamma_2z}{3}}$, where $c_2=\left(\frac{3p_2}{\gamma_2}\right)^\frac{1}{3}$, $p_1=c_2\sum_{j=0}^{k}a_j\left(\frac{\gamma_2}{3}\right)^j$ and $\frac{\gamma_2}{\gamma_1}=3$;
			\end{enumerate}
			where $c_1, c_2 \; \text{and} \; a_j $ $(j=0,1,2,...,k)$ are constants.
		\end{theorem}
		
		The following examples illustrate Theroem \ref{main6}.
		
		\begin{example}
			The function $f(z)=e^{2z}+e^{-2z}$ is the solution of the following differential equation
			\begin{equation*}
				f^2f'-2f'=2e^{6z}-2e^{-6z}
			\end{equation*} satisfies conclusion $(i)$ of the theorem.
		\end{example}
		
		\begin{example}
			The function $f(z)=2e^{3z}$ is the solution of the following differential equation
			\begin{equation*}
				f^2f'+f+3f'=24e^{9z}+20e^{3z},
			\end{equation*} where $p_1=c_1^3*\frac{\gamma_1}{3}=8*3=24$, $\gamma_2=\frac{\gamma_1}{3}=3$ and
			
			$$p_2=c_1\sum_{j=0}^{k}a_j\left(\frac{\gamma_1}{3}\right)^j=2(1*3^0+3*3^1)=20$$
			satisfies conclusion $(ii)$ of the theorem.
		\end{example}
		For the case  $n=1$, the following theorem holds.

		\begin{theorem}\label{t6}
			Suppose that $p_{1}, p_{2}, \gamma_{1}, \gamma_{2}$ are non-zero constants such that $\gamma_{1}\neq\gamma_{2}$ and  $\frac{\gamma_2}{\gamma_1}\neq 2^{\pm 1}$. Let $Q(z,f)$ be a linear differential polynomial in $f$  whose coefficients are small function of $f$ then  the differential equation
			\begin{equation}\label{18}
				f(z)f'(z)+Q(z,f)=p_1e^{\gamma_{1}z}+p_2e^{\gamma_{2}z}
			\end{equation}
			admits no transcendental entire solution of finite order satisfying $N\left(r,\frac{1}{f}\right)=S(r,f)$.
		\end{theorem}
		
		The examples provided below highlight the importance of the conditions assumed in the hypothesis of Theorem \ref{t6}.
		
		\begin{example}
			The function $f(z)=1+\sqrt{3e^z+e^{3z}}$ is the solution of the following differential equation
			\begin{equation}
				f(z)f'(z)+f'(z)=\frac{3}{2}e^z+\frac{3}{2}e^{3z}.
			\end{equation} 
			But $N\left(r,\frac{1}{f}\right)\neq S(r,f)$ for this function.
		\end{example}

		\begin{example}
			The function $f(z)=e^{2z}$ is the solution of the following differential equation
			\begin{equation}
				f(z)f'(z)+f'(z)=2e^4z+2e^{2z}.
			\end{equation} 
			Here $\frac{\gamma_2}{\gamma_1}=\frac{1}{2}$, which  contradicts the hypothesis as $\frac{\gamma_2}{\gamma_1}\neq 2^{\pm{1}}$.
		\end{example}
		In this sequence, a natural question arises  that what will happen if R.H.S.  of equation  \eqref{;} consists of three terms instead of two terms?  We give a positive answer to this question via Theorem\ref{main}.
	\begin{theorem}\label{main}
		Let $n\geq 3$ and $P(z,f)$ be a linear differential polynomial in $f$ and its derivatives whose coefficients are rational functions. Assume that  $P_{j}(z)$ $(j=1,2,3)$ are non-zero rational functions and $\gamma_{j}(z)$ $(j=1,2,3)$   are non-constant polynomials. Also, $\gamma_{1}^{'}(z),\gamma_{2}^{'}(z)$ and $\gamma_{3}^{'}(z)$ are distinct to each other. If $f$ is a transcendental meromorphic solution of the differential equation
		\begin{equation}\label{dedeg2}
			f^{n}(z)f'(z)+P(z,f)=P_{1}(z)e^{\gamma_{1}(z)}+P_{2}(z)e^{\gamma_{2}(z)}+P_{3}(z)e^{\gamma_{3}(z)}
		\end{equation}
		such that $f(z)$ has finitely many poles, then $f$ is of finite order. Moreover, $f(z)$ cannot be represented in the form $f(z)=w(z)e^{t(z)}$, where $w(z)$ is a non-zero rational function and $t(z)$ is a non constant polynomial.
	\end{theorem}	
	Theorem \ref{main} is demonstrated using following examples.
	\begin{example}
		The function $f(z)=\frac{1}{ze^z}+2$ is the solution of the following differential equation
		\begin{equation*}
			f^3f'-f'=-\left(\frac{1}{z^4}+\frac{1}{z^5}\right) e^{-4z}-\frac{3(z^2+1)}{z^5}e^{-3z}-\frac{3(z^2+1)}{z^4}e^{-2z}.
		\end{equation*}
	\end{example}

	-
	\begin{example}
		The function $f(z)=e^z+z$ is the solution of the following differential equation
		\begin{equation*}
			f^3f'-z^3f-3zf-3z^2=e^{4z}+(3z+1)e^{3z}+(3z^2+3z)e^{2z}.
		\end{equation*}
	\end{example}
	\begin{theorem}
		Let $n\geq 3$, $t\geq 0$ and $m\geq 1$ be integers, $n\geq m$ and $P(z,f,f',...,f^{(t)})$ be differential polynomial in $f(z)$ of degree $d\leq n$ with small functions of $f(z)$ as its coefficients. Suppose that $P_i$ and $\alpha_i$ are nonzero constants for $i=1,2,...,m,$ and $|\alpha_1|>|\alpha_2|>\cdots>|\alpha_m|$. If $f(z)$ is a meromorphic solution of the differential equation
		\begin{equation}
			f^nf'+P(z,f,f',...,f^{(t)}) = P_1e^{\alpha_1z}+P_2e^{\alpha_2z}+...+ P_me^{\alpha_m z},
		\end{equation}
		then $f(z)=q_1e^{\frac{\alpha_1z}{n+1}}$, where $q_1$ is a nonzero constant such that $q_1^{n+1}=\frac{(n+1)P_1}{\alpha_1}$ and $\alpha_1, \alpha_2, ..., \alpha_m $ are in one line.
	\end{theorem}
	This theroem is flawed. The following example satisfies all the assumptions of the  theorem; however, the corresponding solution is not of the type described above.
	\begin{example}
		$f^3f'+f^2+1=2e^{8z}+6e^{6z}+5e^{4z}$ \text{has solution} $f(z)=e^{2z}+1$.
	\end{example}
	So, to tackle this flaw, we have the following result.

	Result:
		Let $n\geq3$ be an integer, $Q_d(z,f)$ be a differential polynomial in $f$ of degree $d\leq n-1,\; p_1,\; p_2,\; p_3$ be non-zero distinct constants. If $f(z)$ is a meromorphic solution with $N(r,f)=S(r,f)$ of the following equation:
		\begin{equation}\label{''}
			f^nf'+Q_d(z,f)=p_1e^{\alpha_1z}+p_2e^{\alpha_2z}+p_3e^{\alpha_3z},
		\end{equation}
		then one of the following occurs.\\
		$(1)$ $\alpha_1:\alpha_2:\alpha_3=n+1:k_1:k_2$, $f(z)=\beta_1e^{\frac{\alpha_1z}{n}}$ and $P_d(z,f)=p_2e^{\alpha_2z}+p_3e^{\alpha_3z}$, or \\
		$  \alpha_1:\alpha_2:\alpha_3=k_3:n+1:k_4$, $f(z)=\beta_2e^{\frac{\alpha_2z}{n}}$ and $P_d(z,f)=p_1e^{\alpha_1z}+p_3e^{\alpha_3z}$, where $k_j \in \{1,2,...,d\},$ $j=1,...,4$ are integers and $\beta_1$, $\beta_2$ are non-zero constants.\\
		$(2)$ $\alpha_1:\alpha_2:\alpha_3=n+1:n:k_5$ $(\text{or}\; n:n+1:k_5)$, and  $f(z)=\eta_1e^{\frac{\alpha_1z+\alpha_2z}{2n+1}}+\eta_0$, where $\eta_1 \; \text{and} \; \eta_0$  are non-zero constants  and $k_5\in\{1,2,...,d\}$.\\
		$(3)$  $\alpha_1:\alpha_2:\alpha_3=n+1:-(n+1):k_6$  and  $f(z)=d_1e^{\frac{\alpha_1z}{n}}+ d_2e^{\frac{\alpha_2z}{n}}$, where $d_1 \text{and } d_2$ are non-zero constants and $k_6\in\{1,2,...,d\}$.\\
      We believe that the above result is true based on strong intuition and supporting evidence. Although, we do not currently have a suitable methodology to establish a rigorous proof, the following example provides further support for its validity.
	
\begin{example}
	 For  $n=3,d=2$, $f(z)=e^z+e^{-z}$ is the solution of the following differential equation:
	$$f^3f'-2f''f'+4=e^{4z}-e^{-4z}-4e^{2z}.$$ It  satisfies conclusion $(3)$ of the result.
\end{example}
	\begin{example}
		For  $n=3,d=1$, $f(z)=e^z+1$ is the solution of the following differential equation:
		$$f^3f'-f'=e^{4z}+3e^{3z}+3e^{2z}.$$ It  satisfies conclusion $(2)$ of the result.
	\end{example} 
	If R.H.S of equation \eqref{''} consists of $p_1$,$p_2$,$p_3$ each of which is a small function of $f(z)$, then case $(2)$ yields the solution  
$f(z)=\eta_1e^{\frac{\alpha_1z+\alpha_2z}{2n+1}}+\eta_0(z)$, where $\eta_1$ is a non-zero constant and $\eta_0(z)$ is a small function of $f(z)$. The following example demonstrates the existence of such a solution.
\begin{example}
	
	The function $f(z)=e^z+z$ is a solution of the following differential
	equation
$$	f^3f'+z^2f+z^3f'-z^3=e^{4z}+(z+1)e^{3z}+(z^2+z)e^{2z},$$
	here  $n=3, d=1$.
		\end{example}
The following example shows that conclusion $(1)$ of the result can occur.
	\begin{example}
		For  $n=3,d=2$, $f(z)=e^z$ is the solution of the following diferential equation:
		$$f^3f'+3f^2+2f=e^{4z}+3e^{2z}+2e^{z}.$$ 
	\end{example}


	
	Throughout this paper, we assume that the reader is acquainted with the standard notations of Nevanlinna theory. For a meromorphic function $f$, $n(r,f)$, $N(r,f)$, $\Bar{N}(r,f) $, $m(r,f)$ and  $T(r,f)$ denote un-integrated counting function, integrated counting function, reduced counting function, proximity function and characteristic function respectively, see \cite{hay,ilpo,yanglo}.
	
	In Section $2$, we state  some definitions, lemmas  and results. In Section 3, we prove theorems.
	\medskip
	
	\section{\textbf{Auxiliary results}}
	To maintain self-containment, we briefly recall the definitions of the order of growth 
	$\rho(f)$,  hyper-order of growth $\rho_{2}(f)$  and the exponent of convergence of zeros $\lambda(f)$ for a meromorphic function $f(z)$.
	
	$$\rho(f)=\limsup_{r\to\infty}\frac{\log T(r,f)}{\log r},$$
	$$\rho_2(f)=\limsup_{r\to\infty}\frac{\log \log T(r,f)}{\log r},$$
	and
	$$\lambda(f)=\limsup_{r\to\infty}\frac{\log n\left(r,\frac{1}{f}\right)}{\log r}.$$ 
	
	\bigskip
	
	For a meromorphic function $f$, Nevanlinna’s First Main Theorem asserts  that
	$$T\left(r,\frac{1}{f-a}\right)=T(r,f)+O(1),$$
	for all $a\in \mathbb{C}$,
	where $	O(1)$ denotes a bounded quantity depending only on $a$.
	\smallskip
	
	A meromorphic function $g(z)$ is a small function of $f(z)$ if $T(r,g)=S(r,f)$ and vice versa.
	For a meromorphic function $f(z)$,  $S(r,f)$ denotes the quantity satisfying $S(r,f)=o(T(r,f))$, as $r\to\infty$, outside  of a possible exceptional set $E$ (not necessarily same at each occurrence) of  finite linear measure. \\
	A linear difference polynomial $P$  is an expression formed by finite sum of monomials  involving $f$ and its shifts. It has the general form
	\begin{equation*}
		U^*(z,f)=\sum_{ j=0}^{m}d_jf(z+k_j),
	\end{equation*}
	where $k_j's$ are shifts and $d_j's$ are constants.

	Borel’s Lemma is a fundamental tool in applying Nevanlinna theory to complex differential equations.
	\begin{lemma}\label{imple1}(Borel's Lemma)\cite{cc}
		Suppose that $f_1(z),f_2(z),...,f_n(z)(n\geq2)$ are meromorphic functions and $h_1(z),h_2(z),...,h_n(z)$ are entire functions satisfying:
		\begin{enumerate}
			\item $\sum_{i=1}^{\infty} f_i(z)e^{h_i(z)}\equiv 0$.
			\item For all $1\leq i< k\leq n$, $h_i-h_k$ are  constants.
			\item For each $1\leq i\leq n$ and for all $1\leq m< k\leq n$,  $T(r,f_i(z))=o(T(r,e^{h_m-h_k}))$ as $r\rightarrow \infty$ outside of a set of finite linear measure.
		\end{enumerate}
		Then $f_i\equiv0\ (i=1,2,\ldots,n)$.
	\end{lemma}
	We now state a lemma that gives an estimate for the proximity function of the logarithmic derivative of a meromorphic function $f(z)$. 
	
	\begin{lemma}\label{imple2}\cite{ilpo}
		Suppose $f(z)$ is a transcendental meromorphic function and $k\geq1$ is an integer then 
		\begin{equation*}
			m\left(r,\frac{f^{(k)}}{f}\right)= S(r,f).
		\end{equation*}
		If $f$ is of finite order growth, then
		$$m\left(r,\frac{f^{(k)}}{f}\right)=O(\log r).$$
	\end{lemma}

	The next lemma gives a sharp asymptotic estimate between $T(r, f(z + c))$ and $T(r, f)$, for meromorphic functions of finite order.
	\begin{lemma}\label{imple3}\cite{cf}
		Suppose $f(z)$ is a meromorphic function of finite order $\rho$ and $c$ is a non-zero complex constant then for every $\epsilon>0$,
		\begin{equation*}
			T(r,f(z+c))= T(r,f)+O(r^{\rho-1+\epsilon})+O(\log r).
		\end{equation*}
	\end{lemma}
	Next lemma plays an important role in the study of complex differential-difference
	equations.
	\begin{lemma}\label{imple 4}\cite{yl}
		Let $f(z)$ be  a transcendental meromorphic function, and $P(z,f)$, $Q(z,f)$ be two differential-difference polynomials of $f(z)$. If 
		\begin{equation*}
			f^n(z) P(z,f)=Q(z,f)
		\end{equation*}
		holds, and if the total degree of $Q(z,f)$ in  $f(z)$ and its derivatives and their shifts is at most $n$, then
		\begin{equation*} 
			m(r,P(z.f))=S(r,f),
		\end{equation*}
		for all $r$ outside of a possible exceptional set of finite logarithmic measure.
	\end{lemma}
	
	The next lemma formulates the difference version of the logarithmic derivative lemma for finite-order meromorphic functions.
	\begin{lemma}\label{imple}\cite{cf}
		Suppose $f(z)$ is a meromorphic function with $\rho(f)<\infty$ and $c_1,c_2 \in \mathbb{C}$ such that $c_1\neq c_2$ then for each $ \epsilon>0$, we have
		\begin{equation*}
			m\left(r,\frac{f(z+c_1)}{f(z+c_2)}\right)=O(r^{\rho-1+\epsilon}).
		\end{equation*}
	\end{lemma}
	Next lemma estimates the characteristic function of an exponential polynomial $f$. This lemma can be seen in \cite{whl}.
	\begin{lemma}\label{imple5}\cite{whl}
		Suppose $f$ is an entire function given by
		$$f(z)=B_{0}(z)+B_{1}(z)e^{w_{1}z^{t}}+B_{2}(z)e^{w_{2}z^{t}}+\cdots+B_{m}(z)e^{w_{m}z^{t}},$$
		where $B_{i}(z);0\leq i\leq m$ denote either exponential polynomial of degree $<t$ or polynomial in $z$, $w_{i};1\leq i\leq m$ denote the constants and $t$ denotes a natural number. Then
		$$T(r,f)=C(Co(W_{0}))\frac{r^{t}}{2\pi}+o(r^{t}).$$
		Here $C(Co(W_{0}))$ is the perimeter of the convex hull of the set $W_{0}=\{0,\overline{w}_{1},\overline{w}_{2},...,\overline{w}_{m}\}$.
	\end{lemma}
		%
	
	

	\section{\textbf{Proof of  main theorems}}

\begin{proof}[\textbf{\underline{Proof of Theroem \ref{main2} }}] 
	Let  $f$ be a transcendental entire solution with  $\rho_2(f)<1$. For simplicity, we write $U^*:=U^*(r,f)$ and $\rho(p_1e^{\gamma_{1}z}+p_2e^{\gamma_{2}z})=1$ implies $\rho(f^2f'+U^*)=1$. Now, using equation  \eqref{mde} and Lemma \ref{imple}, it follows that
\begin{align*}
	T(r,p_1e^{\gamma_{1}z}+p_2e^{\gamma_{2}z})=T(r,f^2f'+U^*)&\leq T(r,f^2f')+T(r,U^*)\\
	&=m(r,f^2f')+m(r,U^*)\\
	&\leq 2m(r,f)+m(r,f')+m\left(r,\frac{U^*}{f}\right)+m(r,f)\\
	&\leq 4m(r,f)+S(r,f)=4T(r,f)+S(r,f).
\end{align*} 
This inequality gives $\rho(f)\geq 1$. Now,   by equation \eqref{mde},  $\rho(f^2f'+U^*)=1$, and by the property of characteristic function, we have  $T(r,f^2f'+U^*)=O(r)$, and hence
\begin{align*}
&4T(r,f)+S(r,f)=O(r), \\&  T(r,f)=O(r).
\end{align*}
Along with the definition of order, we have $\rho(f)\leq1$.
Hence, $\rho(f)=1$.\\
Differentiating equation \eqref{mde}, we get
\begin{equation}\label{dmde}
	f^2f''+2ff'^2+(U^{*})'=p_1\gamma_1 e^{\gamma_1z}+p_2\gamma_2 e^{\gamma_2z}.
\end{equation}
Eliminating $e^{\gamma_1z}$ from equations \eqref{mde} and \eqref{dmde}, we obtain
\begin{equation}\label{eqe2}
	\gamma_1f^2f'+\gamma_1 U^{*}-f^2f''-2f(f')^2-(U^{*})'=(\gamma_1-\gamma_2)p_2e^{\gamma_2z}.
\end{equation}
By differentiating equation \eqref{eqe2}, we get
\begin{equation}\label{eqde2}
	2\gamma_1f(f')^2+\gamma_1f^2f''+\gamma_1(U^{*})'-f^2f'''-2ff'f''-2(2f(f')^2f''+(f')^3)-(U^{*})''=p_2\gamma_2(\gamma_1-\gamma_2)e^{\gamma_2z}.
\end{equation}
Eliminate $e^{\gamma_2z}$ from equations \eqref{eqe2} and \eqref{eqde2} to get
\begin{equation}\label{eqq3}
	\phi=H,
\end{equation}
where
\begin{equation}\label{psieq}
	\phi=\gamma_1\gamma_2f^2f'-2(\gamma_1+\gamma_2)f(f')^2-(\gamma_1+\gamma_2)f^2f''+2(f')^3+6ff'f''+f^2f'''
\end{equation}
and 
\begin{equation*}\label{qeq}
	H=(-U^{*})''+(\gamma_1+\gamma_2)( U^{*})'-\gamma_1\gamma_2(U)^*.
\end{equation*}
If $\phi\not\equiv 0$, then using equation \eqref{eqq3} together with  Lemma  \ref{imple 4} and Lemma  \ref{imple2} to obtain
\begin{align}\label{mrfeq3}
	m\left(r,\frac{\phi}{f}\right)&= S(r,f) \; \; \text{as} \;  m\left(r,\frac{H}{f}\right)=S(r,f), \\& \nonumber m\left(r,\frac{\phi}{f^3}\right)=S(r,f).
\end{align}
Using equation \eqref{mrfeq3} and first fundamental theorem of Nevanlinna, we obtain
\begin{align*}
	3T(r,f)=T(r,f^3)&=m\left(r,\frac{1}{f^3}\right) +N\left(r,\frac{1}{f^3}\right)+O(1) \\
	&\leq m\left(r,\frac{\phi}{f^3}\right)+ m\left( r,\frac{1}{\phi}\right) +2N\left( r,\frac{1}{f}\right)+O(1) \\
	&\leq S(r,f)+m(r,\phi)+3N\left( r,\frac{1}{f}\right) \\
	&\leq m\left( r,\frac{\phi}{f}\right) +m(r,f)+3N\left( r,\frac{1}{f}\right)+S(r,f) \\
	&=T(r,f)+3N\left( r,\frac{1}{f}\right)+S(r,f) .
\end{align*}
Therefore
$$2T(r,f)\leq 3N\left( r,\frac{1}{f}\right) +S(r,f),$$
which is one of the conclusions.\\

Now, we assume that $\phi\equiv 0$. Suppose  that $f(z)$ has infinitely many zeros then, in view of equation \eqref{eqq3}, every zero of $f(z)$, except possibly finitely many exceptional ones, has multiplicity at least $2$. Let $z_0$ be a zero of $f(z)$ of multiplicity $k\geq2$. In some small neighbourhood of $z_0$, we may write $f(z)=b_k(z-z_0)^k+b_{k+1}(z-z_0)^{k+1}+\cdots(b_k\neq0,k\geq2)$.
Substituting this expansion into equation \eqref{eqq3} and comparing the lowest-order term in $(z-z_0)$, we obtain
$$2(ka_k)^3+6a_k^3k(k-1)+k(k-1)(k-2)a_k^3=0,$$
which leads to
$$ka_k^3(2k^2+6(k-1)+(k-1)(k-2))=0.$$
This gives $k=0$ or $k=1$, which yields a contradiction from  the assumption  $k\geq2$. Hence, $f(z)$ cannot have infinitely many zeros. Therefore, $f(z)$ possesses only finitely many zeros, and consequently admits the representation where $\tilde {q}(z)$ is any nonzero polynomial  and $p(z)$ is a polynomial of order $1$, say $\tilde{q}(z)=a_tz^t+a_{t-1}z^{t-1}+\cdots+a_0$, and $p(z)=\alpha z+\beta$, where $a_n,a_{n-1},...,a_0,\alpha ,\beta$ are any constants with $a_t\neq0 \;\text{and} \; \alpha\neq 0$. Let $f(z)=q(z)e^{\alpha z}$, where $q(z)=\tilde{q}(z)e^\beta$.
Substituting $f(z)$ in equation \eqref{mde}, we have
\begin{equation*}
	q^2(z)e^{2\alpha z}(q'(z)e^{\alpha z}+e^{\alpha z}\alpha q(z))+U^*(z,f)=p_1e^{\gamma_1z}-p_2e^{\gamma_2z}
\end{equation*}
\begin{equation}\label{req}
	(q^2(z)q'(z)+\alpha q^3(z))e^{3\alpha z}+ \sum_{j=0}^{m}(d_jq_ke^{\alpha z+\tilde{\gamma}})-p_1e^{\gamma_1z}-p_2e^{\gamma_2z}=0,
\end{equation}
where $\tilde{\gamma}=\alpha k_j$ and $ q_k=q(z+k_j)$. 
Equation \eqref{req} reduces to 
\begin{equation}\label{wer}
	A(z)e^{3\alpha z}+B(z)e^{\alpha z}-p_1e^{\gamma_1 z}-p_2e^{\gamma_2 z}=0,
\end{equation}
 where $A(z)= (q^2(z)q'(z)+\alpha q^3(z)),$ and $B(z)=\sum_{j=0}^{m}d_jq_ke^{\tilde{\gamma}}$. Now, we consider the following cases. 
 \begin{enumerate}[(i)]
 	\item If we assume that $3\alpha z \neq \gamma_{i}z $ and $ \alpha z\neq \gamma_{i}z $ for any $i=1,2,$ then by  applying Lemma \ref{imple1} to equation \eqref{wer}, we obtain $p_{1}\equiv 0\equiv p_{2}$. This contradicts with the assumption that $p_1\; \text{and} \;p_2 $ are nonzero constants. 
 	\item If we assume that $3\alpha z \neq \gamma_{i}z $  for any $i=1,2,$ and $ \alpha_ z=\gamma_{i}z$, for some $i=1,2,$ say $\alpha z= \gamma _1z$, then equation \eqref{wer} becomes
 	\begin{align*}
 			A(z)e^{3\alpha z}+B(z)e^{\alpha z}(B(z)-p_1)-p_2e^{\gamma_2 z}=0.
 	\end{align*}
 	Now, applying Lemma \ref{imple1} to the above equation, we have $p_2\equiv 0$, which is not possible.
 	\item If we assume that $3\alpha z =\gamma_{i}z $  for some  $i=1,2,$ say $3\alpha z= \gamma _2z$ and $ \alpha z\neq \gamma_{i}z $  for any $i=1,2,$ then equation \eqref{wer} becomes
 	\begin{align*}
 	(A(z)-p_2)e^{3\alpha z}+B(z)e^{\alpha z}-p_1e^{\gamma_1 z}=0.
 	\end{align*}
  Then applying Lemma \ref{imple1} to the equation \eqref{wer}, we get $p_1\equiv 0$, a contradiction.
 	\item If we assume that $3\alpha z =\gamma_{i}z $ and $ \alpha z=\gamma_{i}z $  for some $i=1,2$,  say $3\alpha z= \gamma _1z$ and $\alpha z= \gamma _2z$,   then equation \eqref{wer} becomes
 	\begin{align*}
 		(A(z)-p_1)e^{3\alpha z}+(B(z)-p_2)e^{\alpha z}=0.
 	\end{align*}
 	Then applying Lemma \ref{imple1} to the equation \eqref{wer}, we get $A(z)=p_1$, and $B(z)=p_2$, which further implies $q(z)=c_1$, some constant $i.e.$ $q^2(z)q'(z)+\alpha q^3(z)$ gives $\alpha c_1^3=p_1$,  and $B(z)=\sum_{j=0}^{m}d_jc_1e^{\tilde{\gamma}}=p_2$. Hence, $f=c_1e^{\frac{\gamma_1z}{3}}$, where $\frac{\gamma_1}{\gamma_2}=3$ and $c_1=\left(\frac{3p_1}{\gamma_1}\right)^\frac{1}{3}$.\\ Also, we can get another  $f=c_2e^{\frac{\gamma_2z}{3}}$, where $c_2=\left(\frac{3p_2}{\gamma_2}\right)^\frac{1}{3}$ and $\frac{\gamma_2}{\gamma_1}=3$.
 \end{enumerate}
 Hence,  the result holds.
\end{proof}
\begin{proof}[\textbf{\underline{Proof of Theroem \ref{main6} }}] 
	We proceed with the same proof as done in Theroem \ref{main2}, then with  the help of  Lemma \ref{imple2} instead of Lemma \ref{imple}, we get $f(z)=q(z)e^{\alpha z}$. Substituting it in main equation, we have 
	\begin{equation}\label{req1}
		A(z)e^{3\alpha z}+ S(z)e^{\alpha z}-p_1e^{\gamma_1z}-p_2e^{\gamma_2z}=0,
	\end{equation} where $(q^2(z)q'(z)+\alpha q^3(z))=A(z)$, $\sum_{j=0}^{k}a_jP_j(q,q',q'',...,q^{(j)})=S(z)$.
	Now, proceed with the four cases as done in previous theorem.
	 \begin{enumerate}[(i)]
		\item If we assume that $3\alpha z \neq \gamma_{i}z $ and $ \alpha z\neq \gamma_{i}z $ for any $i=1,2,$ then by  applying Lemma \ref{imple1} to equation \eqref{req1}, we obtain $p_{1}\equiv 0\equiv p_{2}$. This contradicts with the assumption that $p_1\; \text{and} \;p_2 $ are nonzero constants. 
		\item If we assume that $3\alpha z \neq \gamma_{i}z $  for any $i=1,2$ and $ \alpha_ z=\gamma_{i}z$, for some $i=1,2,$ say $\alpha z= \gamma _1z$, then equation \eqref{req1} becomes
		\begin{align*}
			A(z)e^{3\alpha z}+B(z)e^{\alpha z}(S(z)-p_1)-p_2e^{\gamma_2 z}=0.
		\end{align*}
		Now, applying Lemma \ref{imple1} to the above equation, we have $p_2\equiv 0$, which is not possible.
		\item If we assume that $3\alpha z =\gamma_{i}z $  for some  $i=1,2,$ say $3\alpha z= \gamma _2z$ and $ \alpha z\neq \gamma_{i}z $  for any $i=1,2,$, then equation \eqref{req1} becomes
		\begin{align*}
			(A(z)-p_2)e^{3\alpha z}+S(z)e^{\alpha z}-p_1e^{\gamma_1 z}=0.
		\end{align*}
		Then applying Lemma \ref{imple1} to the above equation, we get $p_1\equiv 0$, a contradiction.
		\item If we assume that $3\alpha z =\gamma_{i}z $ and $ \alpha z=\gamma_{i}z $  for some $i=1,2$,  say $3\alpha z= \gamma _1z$ and $\alpha z= \gamma _2z$,   then equation \eqref{req1} becomes
		\begin{align*}
			(A(z)-p_1)e^{3\alpha z}+(S(z)-p_2)e^{\alpha z}=0.
		\end{align*}
		Then applying Lemma \ref{imple1} to the above equation,  we get $A(z)=p_1$, and $S(z)=p_2$, which further implies $q(z)=c_1$, some constant $i.e.$ $q^2(z)q'(z)+\alpha q^3(z)$ gives $\alpha c_1^3=p_1$,  and
		\begin{align*}
S(z)=&\sum_{j=0}^{k}a_jP_j(q,q',q'',...,q^{(j)})\\=&c_1\sum_{j=0}^{k}a_j\alpha^j=c_1\sum_{j=0}^{k}a_j\left(\frac{\gamma_1}{3}\right)^j=p_2.
	\end{align*}  
	Hence, $f=c_1e^{\frac{\gamma_1z}{3}}$, where $\frac{\gamma_1}{\gamma_2}=3$ and $c_1=\left(\frac{3p_1}{\gamma_1}\right)^\frac{1}{3}$.\\ Also, we can get another  $f=c_2e^{\frac{\gamma_2z}{3}}$, where $c_2=\left(\frac{3p_2}{\gamma_2}\right)^\frac{1}{3}$ and $\frac{\gamma_2}{\gamma_1}=3$.
	\end{enumerate}
	Here, we complete the proof.
\end{proof}
\begin{proof}[\textbf{\underline{Proof of Theroem \ref{t6} }}]
	 We will prove this result by contradiction. Write $Q(z,f):=Q$ and let $f(z)$ be a transcendental entire solution of finite order  of the equation \eqref{18} with $N\left(r,\frac{1}{f}\right)=S(r,f)$. Now, differentiating equation \eqref{18}, we get 
	 
	\begin{equation}\label{19}
		ff''+(f')^2+Q'=p_1\gamma_1 e^{\gamma_1z}+p_2\gamma_2 e^{\gamma_2z}.
	\end{equation}
	Eliminating $e^{\gamma_1z}$ from equations \eqref{18} and \eqref{19}, we obtain
	\begin{equation}\label{20}
		\gamma_1ff'+\gamma_1Q-ff''-f'^2 -Q'=(\gamma_1-\gamma_2)p_2e^{\gamma_2z}.
	\end{equation}
	Now, differentiating equation \eqref{20}, we get
	\begin{equation}\label{21}
		\gamma_1ff''+\gamma_1(f')^2+\gamma_1Q'-ff'''-f'f''-2f'f''-Q''=p_2\gamma_2(\gamma_1-\gamma_2)e^{\gamma_2z}.
	\end{equation}
	From equations \eqref{20} and \eqref{21}, eliminate $e^{\gamma_2z}$ to get
	\begin{equation}\label{eqqq3}
		\psi=K,
	\end{equation}
	where 
	\begin{equation}\label{23}
		\psi =-ff'''+( \gamma_1+ \gamma_2)ff''+(\gamma_1 + \gamma_2)f'^2+ \gamma_1 \gamma_2 ff'-3f'f'',
		\end{equation} and
		 $$K= Q''-(\gamma_1+ \gamma_2)Q'-\gamma_1 \gamma_2 Q.$$
	If $\psi\not\equiv 0 $, then using equation \eqref{eqqq3} together with  Lemma  Lemma  \ref{imple2} and \ref{imple 4}, we obtain
	\begin{align}\label{mrfeq31}
		m\left(r,\frac{\psi}{f}\right)&= S(r,f) \; \; \text{as} \;  m\left(r,\frac{K}{f}\right)=S(r,f), \\& \nonumber \; \text{and}\;\;\; m\left(r,\frac{\psi}{f^2}\right)=S(r,f).
	\end{align}
	Using equation \eqref{mrfeq31} and first fundamental theorem of Nevanlinna, we obtain
	\begin{align*}
		2T(r,f)=&m\left(r,\frac{1}{f^2}\right)+N\left(r,\frac{1}{f^2}\right)+O(1)\\&\leq m\left(r\frac{\psi}{f^2}\right)+m\left(r,\frac{1}{\psi}\right)+ 2N\left(r,\frac{1}{f}\right)+O(1)\\&\leq T(r,\psi)+S(r,f)=m(r,\psi)+S(r,f)\\&\leq m\left(r,\frac{\psi}{f}\right)+m(r,f)+S(r,f)\\&\leq T(r,f)+S(r,f).
	\end{align*}
	Hence, we get $T(r,f)=S(r,f)$, which is a contradiction.\\
	
	Now, we assume that $\psi\equiv 0$. Suppose that  $f(z)$ has infinitely many zeros, then in view of equation \eqref{eqqq3}, every zero of $f(z)$ has multiplicity at least $2$. Let $z_0$ be a zero of $f(z)$ of multiplicity $k\geq2$. In some small neighbourhood of $z_0$, we may write $f(z)=b_k(z-z_0)^k+b_{k+1}(z-z_0)^{k+1}+\cdots(b_k\neq0,k\geq2).$
	Substituting this expansion into equation \eqref{eqqq3} and comparing the lowest-order term in $(z-z_0)$, we obtain
	$$a_k^2k(k-1)(k-2)+3k^2(k-1)a_k^2=0,$$
	which leads to
	$$ka_k^2(3k(k-1)+(k-1)(k-2))=0.$$
	This gives $k=0$ or $k=1$, which yields a contradiction from  the assumption  $k\geq2$. Hence, $f(z)$ cannot have infinitely many zeros. Therefore, $f(z)$ possesses only finitely many zeros, and consequently admits the representation $f(z)=w(z)e^{lz}$ (as we did in Theorem \ref{main2}).
	Substituting $f(z)$ in equation \eqref{18}, we have
		\begin{equation}\label{res1}
		A(z)e^{2l z}+ L(z)e^{l z}-p_1e^{\gamma_1z}-p_2e^{\gamma_2z}=0,
	\end{equation} where $(w(z)w'(z)+lw^2(z))=A(z)$, $\sum_{j=0}^{k}a_jP_j(q,q',q'',...,q^{(j)})=L(z)$.
	 Now, we consider the following cases. For cases $(i), (ii),(iii)$, we obtain similar contradictions as in the Theorem \ref{main2}.
	\begin{enumerate}[(iv)]
	 \item If we assume that $2l z =\gamma_{i}z $ and $ l z=\gamma_{i}z $  for some $i=1,2$,  say $2l z= \gamma _1z$ and $l z= \gamma _2z$,  then equation \eqref{res1} becomes
		\begin{align*}
			(A(z)-p_1)e^{2l z}+(L(z)-p_2)e^{l z}=0.
		\end{align*}
		Then applying Lemma \ref{imple1} to the equation \eqref{res1}, we get $A(z)=p_1$, and $B(z)=p_2$, and $\gamma_1=2l$, and $\gamma_2=l$, a contradiction to the hypothesis. Hence, we have a contradiction for  both the cases $\psi\equiv 0$ and $\psi \not\equiv 0$.
		
	\end{enumerate}
	Hence, the result holds.
	\end{proof}
	
		\begin{proof}[\textbf{\underline{Proof of Theorem \ref{main}}}] 
		Suppose that $f$ is a transcendental meromorphic function satisfying equation \eqref{dedeg2}.  Since $f$ has only finitely many poles, it follows that $f'$ also has only finitely many poles. Therefore, using basic facts of Nevanlinna theory   together with Lemma  \ref{imple5} and Lemma  \ref{imple2} in equation \eqref{dedeg2}, we obtain,
		
		\begin{align*}
			T(r,f^{n})=&T\left(r,\frac{1}{f'}(P_{1}e^{\gamma_{1}(z)}+P_{2}e^{\gamma_{2}(z)}+P_{3}e^{\gamma_{3}(z)}-P(z,f))\right)\\
			&\leq T\left(r,\frac{1}{f'}\right)+T(r,P_{1}e^{\gamma_{1}(z)}+P_{2}e^{\gamma_{2}(z)}+P_{3}e^{\gamma_{3}(z)})+T(r,P(z,f))\\& \qquad  \qquad  \qquad \qquad \qquad \qquad \qquad \qquad \qquad   \qquad  \qquad +\log 2\\
			&\leq  T(r,f)+ T(r,P_{1}e^{\gamma_{1}(z)}+P_{2}e^{\gamma_{2}(z)}+P_{3}e^{\gamma_{3}(z)})+ m\left(r,\frac{P(z,f)}{f}f\right)\\&  \qquad \qquad \qquad \qquad  \qquad \qquad \qquad  \qquad +N(r,P(z,f)) +\log 2\\
			&\leq Rr^{k}+2T(r,f)+S(r,f),
		\end{align*}
		where $$R=\frac{\text{sum of the leading coefficients of} \ \gamma_{1}(z), \gamma_{2}(z)\ \text{and} \ \gamma_{3}(z)}{\pi}$$ and $$k=\max\{\deg \gamma_{1}(z),\deg \gamma_{2}(z),\deg \gamma_{3}(z)\}.$$
		Thus, we have $(n-2)T(r,f)\leq Br^{k}+S(r,f)$. Since $n\geq3$, it follows that $ T(r,f)=O(r^k)$ as $r \to \infty$ and hence $f(z)$ is a meromorphic function of finite order. \\
		It remains to prove the second assertion of the theorem. We proceed by contradiction. Assume $f(z)=w(z)e^{t(z)}$ is a solution of equation \eqref{dedeg2}, where $w(z)$ is a non-zero rational function and $t(z)$ is a non-constant polynomial, then L.H.S of the equation
		\begin{equation*}
			f^{n}(z)f'(z)+P(z,f)=P_{1}(z)e^{\gamma_{1}(z)}+P_{2}(z)e^{\gamma_{2}(z)}+P_{3}(z)e^{\gamma_{3}(z)}
		\end{equation*} 
		becomes \begin{align}
			&(w(z)e^{t(z)})^{n}(w(z)e^{t(z)})'+\sum_{i=1}^{s}(b_{i}(z)f^{(i)}(z))+{b_{0}}(z)\nonumber\\&=w^n(z)(e^{(n+1)t(z)})(w'(z)+t'(z)w(z))+\sum_{i=0}^{l}(b_{i}(z)(w(z)e^{t(z)})^{(i)})+{b_{0}}(z) .
		\end{align}
		This implies 
		\begin{align}\label{eqlast}
			e^{(n+1)t(z)}\left(w^{n}(z)w'(z)+w^{n+1}(z)t'(z)\right)+&S(z)e^{t(z)}+{b_{0}}(z)\\&\nonumber-P_{1}(z)e^{\gamma_{1}(z)}-P_{2}(z)e^{\gamma_{2}(z)}-P_{3}(z)e^{\gamma_{3}(z)}=0,
		\end{align}
		where $S(z)=b_{0}(z)w(z)+b_{1}(z)[w^{'}(z)+w(z)t^{'}(z)]+\cdots+b_{s}(z)[w^{s}(z)+\cdots+w(z)(t^{'}(z))^{s}]$ is a rational function. We proceed to analyze the following cases:
		
		\begin{enumerate}[(i)]
			\item If we assume that $t(z)-\gamma_{j}(z)\neq $ constant, and $(n+1)t(z)-\gamma_{j}(z)\neq $ constant, for any $j=1,2,3,$ then by  applying Lemma \ref{imple1} to equation \eqref{eqlast}, we obtain $P_{1}(z)\equiv 0\equiv P_{2}(z)=P_{3}(z)$. This contradicts with the assumption that $P_1(z), \;P_2(z), \; \text{and} \;P_3(z)$ are nonzero rational functions. 
			\item Assume that $t(z)-\gamma_{j}(z)=$ constant for some $j=1,2,3$, say $t(z)-\gamma_{1}(z)=$ constant, and $(n+1)t(z)-\gamma_{j}(z)\neq$ constant, for any $j=1,2,3$. Let $t(z)=c_{r}z^{r}+c_{r-1}z^{r-1}+\cdots+c_{0}$ and $\gamma_{1}(z)=c_{r}z^{r}+c_{r-1}z^{r-1}+\cdots+d_{0}$, then $t(z)-\gamma_{1}(z)=c_{0}-d_{0}$ and equation \eqref{eqlast} becomes
			\begin{align*}
				e^{(n+1)t(z)}\left(w^{n}(z)w'(z)+w^{n+1}(z)t'(z)\right)&+(S(z)e^{c_{0}}-P_{1}(z)e^{d_{0}})e^{\tilde{\gamma_{1}}(z)}+\\& {b_{0}}(z)-P_{2}(z)e^{\gamma_{2}(z)}-P_{3}(z)e^{\gamma_{3}(z)}=0,
			\end{align*}where $\tilde{\gamma_{1}}(z)=c_{r}z^{r}+c_{r-1}z^{r-1}+\cdots+c_{1}z.$
			Now, applying Lemma \ref{imple1} to the above equation, we have $P_{2}(z)\equiv 0\equiv P_{3}(z)$, which is not possible.
			\item If $t(z)-\gamma_{j}(z)\neq$ constant for any $j=1,2,3$,  and $(n+1)t(z)-\gamma_{k}(z)=$ constant, for some $k$, say $nq(z)-\gamma_{2}(z)=$ constant, then applying Lemma \ref{imple1} to  the equation \eqref{eqlast}, we arrive at the same contradiction as in the preceding case.
			\item If $t(z)- \gamma_{j}(z)=$ constant for some $j=1,2,3$, say $t(z)-\gamma_{2}(z)$, and $(n+1)t(z)- \gamma_{j}(z)=$ constant, for some $j=1,2,3$, say $(n+1)t(z)-\gamma_{1}(z)=$ constant, then applying Lemma \ref{imple1} to the equation \eqref{eqlast}, we get $P_3(z)\equiv0$, a contradiction.
			
		\end{enumerate}
		This completes the proof.
	\end{proof}

\end{document}